\documentclass[11pt]{article}
\usepackage{amssymb,amsmath,amsthm}
\usepackage{multirow}
\usepackage{enumerate}
\usepackage[dvipsnames]{xcolor}
\usepackage{pgf,tikz}
\usepackage{hyperref}
\usetikzlibrary{arrows}

\let\oldmarginpar\marginpar
\renewcommand\marginpar[1]{\-\oldmarginpar[\raggedleft\footnotesize #1]%
{\raggedright\footnotesize #1}}

\newtheorem{theorem}{Theorem}
\newtheorem{corollary}[theorem]{Corollary}
\newtheorem{lemma}[theorem]{Lemma}
\newtheorem{proposition}[theorem]{Proposition}
\newtheorem{obs}[theorem]{Observation}

\theoremstyle{definition}

\newtheorem{definition}[theorem]{Definition}

\newtheorem{remark}[theorem]{Remark}

\numberwithin{equation}{section}

\renewcommand{\deg}{d}

\newcommand{\girth}{\mathop{\mathrm{girth}}}

\newcommand{\Nnn}{\mathbb{N}}

\newcommand{\vanish}[1]{}

\begin{document}

\title{Reconfiguration graphs of shortest paths\thanks{This project was initiated as part of the REUF program at AIM, NSF grant DMS 1620073.}
\\
}

\author{
John Asplund\\
{\small Department of Technology and Mathematics,} \\
{\small Dalton State College,} \\
{\small Dalton, GA 30720, USA} \\
{\small jasplund@daltonstate.edu}\\
\\
Kossi Edoh\\
{\small Department of Mathematics} \\
{\small North Carolina Agricultural and Technical State University} \\
{\small Greensboro, NC 27411, USA} \\
{\small kdedoh@ncat.edu} \\
\\
Ruth Haas \thanks{Partially supported by Simons Foundation Award Number 281291}\\
{\small Department of Mathematics,} \\
{\small University of Hawaii at Manoa,}\\
{\small Honolulu, Hawaii 96822, USA}\\
{\small rhaas@hawaii.edu}\\
{\small and Smith College, Northampton MA 01063}\\
\\
Yulia Hristova\\
{\small Department of Mathematics and Statistics,}\\
{\small University of Michigan - Dearborn,}\\
{\small Dearborn, MI 48128, USA}\\
{\small yuliagh@umich.edu}\\
\\
Beth Novick\\
{\small Department of Mathematical Sciences,} \\
{\small Clemson University,} \\
{\small Clemson, SC 29634, USA} \\
{\small nbeth@clemson.edu}\\
\\
Brett Werner\\
{\small Department of Mathematics, Computer Science \& Cooperative Engineering,} \\
{\small University of St. Thomas,} \\
{\small Houston, TX 77006, USA} \\
{\small wernerb@stthom.edu}\\
\\
 }

\maketitle

\begin{abstract}
For a graph $G$   and $a,b\in V(G)$,  the  shortest path reconfiguration graph of $G$ with respect to $a$ and $b$ is denoted by $S(G,a,b)$. The vertex set of $S(G,a,b)$ is the set of  all shortest paths between $a$ and $b$ in $G$.
Two vertices in $V(S(G,a,b))$ are adjacent,  if their corresponding paths in $G$ differ by exactly one vertex.
This paper examines the
 properties of shortest path graphs. Results include establishing classes of graphs that appear as shortest path graphs, decompositions and sums involving shortest path graphs, and the complete classification of shortest path graphs with girth $5$ or greater.  We also show that the shortest path graph of a grid graph is an induced subgraph of a lattice.
\end{abstract}

\section{Introduction}

The goal of reconfiguration problems is to determine whether it is possible to transform one feasible solution $s$ into a target feasible solution $t$ in a step-by-step manner (a reconfiguration) such that each intermediate solution is also feasible.
Such transformations can be studied via the reconfiguration graph, in which the vertices represent the feasible solutions and there is an edge between two vertices when it is possible to get from one feasible solution to another in one application of the reconfiguration rule.
Reconfiguration versions of vertex coloring \cite{BJLPP,BC,CJV,CJV2,CJV3}, independent sets \cite{HD,IDHPSUU,KMM}, matchings \cite{IDHPSUU}, list-colorings \cite{IKD}, matroid bases \cite{IDHPSUU}, and subsets of a (multi)set of numbers \cite{EW}, have been studied.
 This paper concerns the reconfiguration of shortest paths in a graph.

\begin{definition}   Let $G$ be a graph with distinct vertices $a$ and $b$. The \emph{shortest path graph} of $G$ with respect to $a$ and $b$ is the graph $S(G,a,b)$ in which every vertex $U$ corresponds to a shortest path in $G$ between $a$ and $b$, and two vertices $U,W \in V(S(G,a,b))$ are adjacent if and only if their corresponding paths in $G$ differ in exactly one vertex. 
\end{definition}

While there have been investigations into shortest path reconfiguration in  \cite{B,KMM2,KMM}, these papers focused on the complexity of changing one shortest path into another\footnote{The shortest path graph is denoted  by SP$(G,a,b)$ in \cite{B}.}. It was found in \cite{B} that the complexity of this problem is PSPACE-complete.  In contrast, the focus of our work is on the structure of shortest path graphs, rather than algorithms. Our main goal is to understand which graphs occur as  shortest path graphs. 
A similar study on classifying color reconfiguration graphs can be found in \cite{beier2016classifying}.

The paper is organized as follows.
Some definitions and notations are provided in Section~\ref{notation}.   Section~\ref{general}  contains some useful properties and examples. 
In particular, we show that paths and complete graphs are shortest path graphs.
 In Section~\ref{one-two_sum}  we show that the family of shortest path graphs is closed under  disjoint union and under Cartesian products.
We establish  a decomposition result  which suggests that, typically, $4$-cycles are prevalent in shortest path graphs.  Thus, we would expect the structure of shortest path graphs containing no $4$-cycles to be rather simple. This is substantiated in Section~\ref{girth5},  where we give a remarkably simple characterization of
shortest path graphs with girth $5$ or greater.  In the process of establishing this characterization, we show that 
 the claw and the odd cycle $C_k$,  for $k>3$ are, in a sense, forcing structures. As a consequence, we determine precisely which cycles are shortest path graphs;  that the claw, by itself, is not a shortest path graph;  and that a tree cannot be a shortest path graph unless it is a path.

 In contrast, our main theorem in the final section of the paper involves a class of shortest path graphs which contain many $4$-cycles. We establish that the shortest path graph of a grid graph is an induced subgraph of the lattice. One consequence of our construction is that the shortest path graph of the hypercube $Q_n$ with respect to two diametric vertices is a Cayley graph on the symmetric group $S_n$.

\setcounter{section}{1}
\section{Preliminaries}\label{notation}

Let $G$ be a graph with distinct vertices $a$ and $b$.  A shortest $a,b$-path in $G$ is a path between $a$ and $b$ of length $d_G(a,b)$.  When it causes no confusion, we write $d(a,b)$ to mean $d_G(a,b)$.
 We often refer to a shortest path as a geodesic and to a shortest $a,b$-path as an $a,b$-{\em geodesic}. Note that any subpath of a geodesic is a geodesic.

If the paths corresponding to  two adjacent  vertices $U, W$ in $S(G,a,b)$ are $a v_1\cdots v_{i-1}v_i v_{i+1} \cdots v_pb$ and
$av_1\cdots v_{i-1}v_i'v_{i+1}\cdots v_pb$, we say that $U$ and $W$
differ in the $i^{{\rm th}}$ index, or that $i$ is the \textit{difference index} of the edge $UW$. We call the  graph $G$  the \textit{base graph} of $S(G,a,b)$, and we say that a graph $H$ is a shortest path graph, if there exists a graph $G$ with $a,b\in V(G)$ such that $S(G,a,b)\cong H$.
Several examples are given in Figure~\ref{moreExamples}.
With a slight abuse of notation, a label for a vertex in the shortest path graph will often also represent the corresponding path in its  base graph.
To avoid confusion between vertices in $G$ and vertices in $S(G,a,b)$, throughout this paper,  we will use lower case letters to denote vertices in the base graph,  and upper case letters to denote vertices in $S(G, a,b)$.

It can  easily be seen that several base graphs can have the same shortest path graph. For example,  if  $e\in E(G)$ and $e$ is an edge not in any
$a,b$-geodesic, then $S(G, a, b) \cong S(G\setminus e, a, b)$.
To this end, we  define the \textit{reduced graph,} $(G,a,b)$,  to be the graph obtained from $G$ by deleting  any edge or vertex that does not occur in
any $a,b$-geodesic,  and contracting any edge that occurs  in all $a,b$-geodesics.
If the reduced graph $(G, a, b)$ is again $G$ then
$G$ is  called a  {\em reduced graph} with respect to $a, b$.
 We may omit the reference to $a,b$ when it is clear from context. 

\begin{figure}[htb]
\begin{center}
\includegraphics[scale=1]{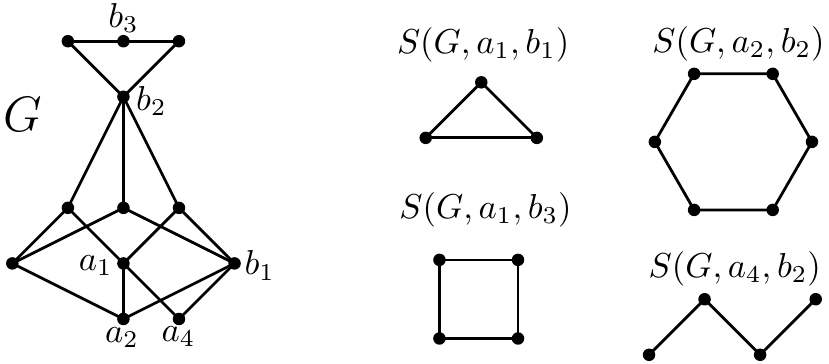}
\end{center}
\caption{Base graph $G$ (left) with several shortest path graphs (right).
}\label{moreExamples}
\end{figure}

We conclude this section with a review of  some basic definitions.
If $G_1$ and $G_2$ are graphs then $G_1\cup G_2$ is defined to be the graph whose vertex set is $V(G_1)\cup V(G_2)$ and whose edge set is $E(G_1) \cup E(G_2)$.  When $V(G_1)\cap V(G_2) =\varnothing$ we say that $G_1$ and $G_2$ are disjoint, and  refer to $G_1\cup G_2$ as the disjoint union of $G_1$ and $G_2$.
For two graphs $G_1$ and $G_2$, the Cartesian product $G_1\,\square\, G_2$ is a graph with vertex set $V(G_1)\times V(G_2)$ and edge set
$\{(v_1,v_2)(u_1,u_2)\,:\, v_i,u_i\in V(G_i) \text{ for } i\in\{1,2\} \text{ and either }v_1=v_2\text{ and } u_1\sim u_2, \text{ or } v_1\sim v_2 \text{ and } u_1=u_2\}.$
If $U_1$ is a $v_0,v_{\ell}$-path and $U_2$ is a $v_{\ell},v_m$-path, where $U_1$ and $U_2$ have only one vertex in common,  namely $v_{\ell}$, then the concatenation of $U_1$ and $U_2$ is the $v_0,v_m$-path $U_1\circ U_2=v_0v_1 \ldots v_{\ell} v_{\ell+1} \ldots v_m $.
A hypercube of dimension $n$,  denoted $Q_n$, is the graph formed by labeling a vertex with each of the $2^n$ binary sequences of length $n$, and joining two vertices with an edge if and only if their sequences differ in exactly one position.

\section{General Properties, Examples,  and Constructions}\label{general}

In this section we answer some natural questions as to which classes of graphs are shortest path graphs.
We easily see that the  empty graph is a shortest path graph. We show that paths and complete graphs are shortest path graphs as well.

\begin{proposition}\label{emptyGraph}
Let $G$ be a graph formed by joining  $t$  paths of equal length greater than $2$,  each having the same end vertices,  $a$ and $b$,  and with all other vertices between any two paths being distinct. Then $S(G,a,b)=\overline{K_t}$.
\end{proposition}

Before finding shortest path graphs with edges, we make
the following simple observation  which will  be used implicitly throughout.

\begin{obs}\label{fundObs}
Let $H$ be a shortest path graph. If $U_1U_2U_3$ is an induced path  in $H$ then $U_1U_2$ and $U_2U_3$ have distinct difference indices.
\end{obs}

We use this observation to construct a family of graphs whose shortest path graphs are paths.

\begin{lemma}\label{firstpath}
For any $k\geq 1$, the path $P_k$ is a shortest path graph.
\end{lemma}
\begin{proof}
 For $k\ge 1$,  define the graph  $G_k$  by
$V(G_k)= \{a,b,v_0, v_1, \ldots , v_{\lfloor k/2 \rfloor}, v_0',v_1', \ldots v'_{\lceil k/2 \rceil}\}
                                              $ and
$E(G_k)= \{av_i,  v_iv_i'\,:\, 0\le i \le \lfloor k/2 \rfloor\} \cup \{v_{i-1}v_i' \,:\, 1\le i \le \lceil k/2 \rceil\}\cup
      \{v_i'b\,:\, 0\le i \le \lceil k/2 \rceil \}.
$
One checks that  $P_k \cong S(G_k,a,b)$.
\end{proof}

\begin{lemma}  For any $n\geq 1$ the complete graph $K_n$ is a shortest path graph.
\end{lemma}
\begin{proof}
Let $a$ and $b$ be the vertices on one side of the bipartition of $K_{2,n}$. Then $S(K_{2,n}, a,b) \cong K_n$.
\end{proof}

In fact, as we show in the proof of  Theorem \ref{completeGraph}, any graph  for which the shortest path graph is a complete graph must reduce to  $K_{2,n}$.    We will see later that in general  there can be different reduced graphs that have the same shortest path graph.

\begin{theorem}\label{completeGraph}
$S(G,a,b)=K_n$ for some $n\in \Nnn$,  if and only if each pair of $a,b$-geodesics in $G$ differs at the same index.
\end{theorem}

\begin{proof}
If every pair of $a,b$-geodesics differs only at the $i^{{\rm th}}$ index, for some $i$, then it is clear that $S(G,a,b)=K_n$,  where $n$ is the number of $a,b$-geodesics in $G$. 
Now suppose that $U,V$ and $W\in V(S(G,a,b))$ are such that $UV$ and $VW$ have distinct difference indices. Then the paths $U$ and $W$ differ at two vertices, so there could be no edge between them in $S(G,a,b)$.
Hence  the reduced graph of $G$ is $K_{2,n}$.
\end{proof}

It is clear that if two graphs give the same reduced graph with respect to a, b then they have the same shortest path graph. It will be   useful to be able to construct different graphs with the same reduced graph.
The next result  involves, in a sense, an operation which is the reverse  of forming a  reduced graph.

\begin{proposition}\label{manybasegraphs}
If $H=S(G,a,b)$ and $d_G(a,b) = k$, then for any $k'\geq k$ there exists a graph $G'$ with vertices $a, b'\in G'$ such that $d_{G'}(a, b') = k'$ and $H\cong S(G',a,b')$.
\end{proposition}

\begin{proof} Suppose $H= S(G, a, b)$. Define $G'$ as follows:  $V(G')= V(G) \cup \{x_1, x_2, \dots x_{k'-k-1}, b'\}$, and
$E(H') = E(H)\cup \{ bx_1, x_1x_2, \dots  x_{k'-k-2} x_{k'-k-1},  x_{k'-k-1}b' \}$.  It is clear that $H\cong S(G', a, b')$.
\end{proof}

\section{Decompositions and Sums}\label{one-two_sum}

In the previous section  we  constructed a few special classes of shortest path graphs.  In the present section  we establish two   methods of obtaining new  shortest path graphs from old.  In particular,  we show that the family of shortest path graphs is closed under  disjoint unions and is  closed under  Cartesian products.

\begin{theorem}\label{disconnect}  
If  $H_1$ and $H_2$  are  shortest path graphs, then $H_1\cup H_2$ is a shortest path graph.
\end{theorem}

\begin{proof}
By Proposition~\ref{manybasegraphs} we can choose  disjoint base graphs $G_i$ for $H_i$, $i\in \{1,2\}$, such
that $\{a_i,b_i\}\in G_i$, with $d_{G_1}(a_1,b_1) = d_{G_2}(a_2,b_2)$ 
 and with $H_i \cong S(G_i,a_i,b_i)$.
Construct a graph $G$ as follows. Let $V(G) = V(G_1) \cup V(G_2) \cup \{a, b\}$ and
 $E(G) = E(G_1) \cup E(G_2) \cup \{aa_1,aa_2, b_1b, b_2b\}$.
It is clear by the construction of $G$ that every $a,b$-geodesic corresponds to an $a_1,b_1$-geodesic through $G_1$ or an $a_2,b_2$-geodesic through $G_2$. In addition, if two shortest paths are adjacent in $S(G_i,a_i,b_i)$, $i \in \{1,2\}$, they are still adjacent in $S(G,a,b)$.
If $U_1$ and $U_2$ are $a,b$-geodesics in $G$ between $a$ and $b$ where $V(U_2)\cap V(G_1)\neq\varnothing $ and $V(U_1)\cap V(G_2)\neq\varnothing$,  then since $a_1, b_1\in V(U_1)$  and $a_2, b_2\in V(U_2)$ we have $U_1\not\sim U_2$. 
Thus the result holds.
\end{proof}

Proposition \ref{concat} concerns  the structure of the  subgraph of a shortest path graph $H\cong S(G,a,b)$ induced by all $a,b$-geodesics containing a given  vertex $v$. 

\begin{proposition}\label{concat}
 Let $G$ be a connected graph with $a, b\in V(G)$ and   $d=d(a,b) \ge 2$.   Let  $H=S(G,a,b)$  and let $v$ be a vertex of $G$ which is on at least one $a,b$-geodesic.  Let $H'$ be the subgraph of $H$  induced by all  vertices corresponding to $a,b$-geodesics which contain $v$.  Then
\[ H' \cong S(G, a, v)\square S(G, v,b).\]
Furthermore,   if  $G_1$  is any subgraph of   $G$  containing  all $a,v$-geodesics,   and   $G_2$  is  any subgraph of $G$  containing  all $v,b$-geodesics, then
$ H'$ is isomorphic to $S(G_1,a,v) \square S(G_2,v,b).$
\end{proposition}

\begin{proof}
Let $H'$ be the subgraph of $H$ induced by all elements of $V(H)$ corresponding to $a,b$-geodesics in $G$ which contain the vertex $v$.
These $a,b$-geodesics are precisely the concatenations $T_{av}\circ T_{vb}$ where $T_{av}$ is an $a,v$-geodesic and $T_{vb}$ is
a $v,b$-geodesic.  Hence we speak interchangeably about the elements in the vertex set of $H'$ and geodesic paths in $G$ of the form
 $T_{av}\circ T_{vb}$.

By definition, the vertex set of $S(G,a,v) \,\square\, S(G,v,b)$ is the collection of ordered pairs $(T_{av},T_{vb})$ where $T_{av}$ is a vertex of $S(G,a,v)$, and
$T_{vb}$ is a vertex of $S(G,v,b)$.
It is clear that the mapping $f: V(H') \rightarrow V(S(G,a,v) \, \square \, S(G,v,b))$ given by $f(T_{av}\circ T_{vb}) = (T_{av},T_{vb})$, is a bijection.  We claim that this bijection is edge-preserving.  Indeed,
let $(U_1,R_1)\sim (U_2,R_2)$  in $S(G,a,v) \,\square\, S(G,v,b)$. Then by definition of Cartesian product,  either
(i) $U_1 \sim U_2$ in $S(G,a,v)$ and $R_1 = R_2$ or (ii) $U_1 = U_2$ and $R_1 \sim R_2$ in $S(G,v,b)$.
In the former case  $U_1$ and $U_2$ differ in exactly one index while $V(R_1)=V(R_2)$, so $U_1\circ R_1 \sim U_2\circ R_2$ in $H'$. 
An analogous argument holds in case (ii).
Now assume that $(U_1,R_1)$ is neither equal to nor adjacent to $(U_2,R_2)$  in $S(G,a,v) \,\square\, S(G,v,b)$.
Then one of the following occurs: $U_1=U_2$, in which case  $R_1$ and $R_2$ differ in at least two indices; $R_1=R_2$ in which case $U_1$ and $U_2$ differ in at least two indices;  or
$U_1\not= U_2$ and $R_1\not= R_2$ in which case $U_1$ and $R_1$ differ with $U_2$ and $R_2$ in at least one index respectively.  In each of these cases $R_1\circ U_1$ and $R_2\circ U_2$ differ in at least two indices and hence are not adjacent, as required.

To complete the proof, we simply note that $S(G,a,v)\cong S(G_1,a,v)$ and that $S(G,v,b)\cong S(G_2,v,b)$.  \end{proof}

For two graphs $G_1$ and $G_2$ with vertex sets such that $V(G_1)\cap V(G_2)=\{c\}$,  the {\em one-sum of $G_1$ and $G_2$ }is  defined to be the graph $G$ with vertex set $V(G_1)\cup V(G_2)$ and edge set $E(G_1)\cup E(G_2)$.  Theorem \ref{onesum} characterizes the shortest path graph of the one-sum of two graphs.

\begin{theorem}\label{onesum}
Let $G_1$ and $G_2$ be graphs with vertex sets such that $V(G_1)\cap V(G_2)=\{c\}$. Let $G$ be the one-sum of $G_1$ and $G_2$.
Then for any $a\in V(G_1)\setminus\{c\}$ and any $b\in V(G_2)\setminus\{c\}$,
$$S(G,a,b)\cong S(G_1,a,c)\,\square\, S(G_2,c,b).$$
\end{theorem}

\begin{proof}
Because $c$ is a cut-vertex, every $a,b$-geodesic in $G$ must contain $c$. 
The result now follows immediately from Proposition~\ref{concat}.
\end{proof}

\begin{corollary}\label{cartesian_closed}
Let  $H_1$ and $H_2$ be shortest path graphs.   Then $H_1 \square H_2$ is also a shortest path graph.
\end{corollary}

\begin{proof}
Let $H_1 = S(G_1,a_1,b_1)$ and $H_2 = S(G_2,a_2,b_2)$,  where $G_1$ and $G_2$ are reduced graphs. Identify $b_1$ with $a_2$ to obtain a graph $(G,a_1,b_2)$ for which $H_1 \square H_2$ is the shortest path graph.
\end{proof}

The construction of Theorem \ref{onesum} leads to  a family of graphs whose shortest path graphs are hypercubes.
Let $J_k$ be the graph formed by taking one-sums of $k$ copies of $C_4$ as follows. For $i=1,\ldots, k$ let $a_i$ and $b_i$ be antipodal vertices in the $i^{{\rm th}}$ copy of  $C_4$.  Form $J_k$  by identifying $b_i$ and $a_{i+1}$ for $i=1, \ldots , k-1$. See Figure \ref{hypercubes}.

\begin{corollary}\label{Cube}
For $J_k$ as defined above, $S(J_k,a_1,b_k)\cong Q_{k}$ where $Q_{k}$ is a hypercube of dimension $k$.
\end{corollary}

\begin{figure}[htb]
\begin{center}
 \includegraphics[scale=1]{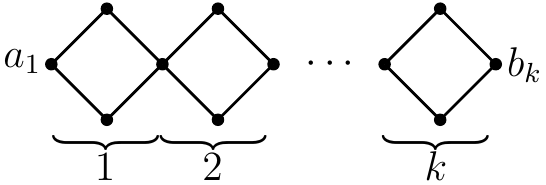}
\end{center}
\caption{Base graph $J_k$ whose shortest path graph is a hypercube.}\label{hypercubes}
\end{figure}

\begin{proof}
For $k=1$, $S(G,a_1,b_1)\cong P_1\cong Q_1$. The proof follows by induction on $k$ and from the statement and proof of  Corollary \ref{cartesian_closed}.
\end{proof}

The next result gives a decomposition of a shortest path graph into a disjoint set of one sums with additional edges.
Note that  the following theorem, holds for all  $1\le i< d(a,b)$. Hence, there are actually $d(a,b)- 1$  different decompositions of this sort.

\begin{theorem}\label{decomp}
Let $H$ be a shortest path graph with reduced   base graph $(G,a,b)$,  where $d(a,b)\ge 2$.
Fix an index $i$,  with $1\le i < d(a,b)$. 
Let $\{v_{i_1},v_{i_2}, \ldots , v_{i_k}\}$ be the set of $k$ vertices in $V(G)$ of distance $i$ from $a$, and let $E_i$ be the set of all edges $UW \in E(H)$ having difference index $i$. Then
\begin{itemize}
\item[(a)] the result of deleting the edges $E_i$ from $H$ yields a graph having $k$ disjoint components, each of which is a Cartesian product:

\[\displaystyle  H\setminus E_i = \bigcup_{j=1}^k  D_{i_j}, \]
where $D_{i_j}=S(G,a,v_{i_j})\square S(G,v_{i_j}, b)$ and

\item[(b)]
For any two subgraphs $D_{i_j}$ and $D_{i_\ell}$, the edges in $E_i$ between $V(D_{ij})$ and $V(D_{i \ell})$
form a partial matching.
\end{itemize}
\end{theorem}

\begin{proof}  For each $j\in\{1, 2, \ldots, k\}$,  it follows from  Proposition~\ref{concat}  that the set of vertices in $H$  corresponding to $a,b$-geodesics containing $v_{i_j}$   induce a subgraph isomorphic to
$S(G,a,v_{i_j})\square S(G,v_{i_j},b)$.   Since each $a,b$-geodesic in $G$ contains precisely one vertex $v_{i_j} \in \{v_{i_1}, \ldots ,v_{i_k}\}$,  these $k$ induced subgraphs  of $H$ are vertex disjoint.  Furthermore,  any pair $U$, $W$ of adjacent vertices in $H$ whose corresponding $a,b$-geodesics contain distinct vertices in $\{v_{i_1}, \ldots ,v_{i_k}\}$,  must differ in index $i$.  We conclude that $UW$ has difference index $i$ and is in $E_i$. This establishes part (a).

Each vertex $U$ in $D_{i_j}$  corresponds to a path with vertex $v_{i_j}$ at the $i^{\rm th}$ index, and each vertex $W$ in $D_{i_\ell}$ corresponds to a path with vertex $v_{i_\ell}$ at the $i^{\rm th}$ index.
Thus for each vertex $U$ in $D_{i_j}$ there is at most one vertex in $D_{i_\ell}$ adjacent to $U$. 
\end{proof}

Note that $4$-cycles occur very often in Cartesian products: Take any edge $UW$  in $H_1$  and any edge $XY$  in $ H_2$.   Then the set of vertices $\{ (U,X), (U,Y), (W,X), (W,Y)\}$  induces a $4$-cycle in $H_1 \square H_2$.
From Theorem \ref{decomp}, part (a), and the fact that  4-cycles are ubiquitous in Cartesian products of  graphs, we conclude that
\begin{obs}\label{prevalent}$4$-cycles are prevalent in shortest path graphs.
\end{obs}
In view of the fact that Theorem \ref{decomp} holds for every  index $i$, Observation~\ref{prevalent} is especially strong:  we expect shortest path graphs having no $4$-cycles to have a relatively  simple structure, and we predict  the study of shortest path graphs with no such restriction to be more challenging.

We conclude this section with another way to combine base graphs. Let $G_1$ and $G_2$ be graphs with edge sets such that $E(G_1)\cap E(G_2) = \{e\}$,  where $e=xy$,  and $V(G_1)\cap V(G_2) = \{x,y\}$. The {\em two-sum of $G_1$ and $G_2$} is defined to be  the graph $G$ with vertex set $V(G_1)\cup V(G_2)$ and edge set $E(G_1)\cup E(G_2)$.  Theorem~\ref{twosum} characterizes the shortest path graph of the two-sum of two graphs.

\begin{theorem}\label{twosum} Let $G_1$ and $G_2$ be graphs with edge sets such that $E(G_1)\cap E(G_2) = \{e\}$,  where $e=xy$,  and $V(G_1)\cap V(G_2) = \{x,y\}$.
Let $G$ be the  two-sum of $G_1$ and $G_2$.
Let $a\in V(G_1)$ and $b\in V(G_2)$ where $\{a,b\} \cap \{x,y\} = \varnothing$.  Then $S(G, a,b)$ is isomorphic to one of the following:
\begin{itemize}
\item[$(i)$] the disjoint union $S(G_1,a,x)\,\square\, S(G_2,x,b)\, \bigcup \, S(G_1,a,y)\,\square\, S(G_2,y,b)$ plus additional edges which comprise a matching between the two,
in the case that $d(a,x)=d(a,y)$ and $d(x,b)=d(y,b)$;
\item[$(ii)$] $S(G_1,a,x)\,\square\, S(G_2,x,b)$, in the case that $d(a,x)\le d(a,y)$ and $d(x,b) < d(y,b)$, \\   or $d(a,x)<d(a,y)$ and $d(x,b) \le d(y,b)$;
\item[$(iii)$] $S(G_1,a,y)\,\square\, S(G_2,y,b)$, in the case that  $d(a,y)\le d(a,x)$ and $d(y,b) < d(x,b)$, \\  or $d(a,y)<d(a,x)$ and $d(y,b) \le d(x,b)$;
\item[$(iv)$] and otherwise $S(G_1,a,x)\,\square\, S(G_2,x,b)\, \bigcup \,S(G_1,a,y)\,\square\, S(G_2,y,b)$,  where vertices common to \\$S(G_1,a,x)\,\square\, S(G_2,x,b)$ and
$S(G_1,a,y)\,\square\, S(G_2,y,b)$  correspond precisely to $a,b$-geodesics containing the edge $e$.
\end{itemize}
\end{theorem}

\begin{proof}
Note that every $a,b$-geodesic has non-empty intersection with  $\{x,y\}$.

\noindent \textbf{Case (i)}  Suppose,  $i=d(a,x)=d(a,y)$ and $d(x,b)=d(y,b)$.  In this case, the vertices $x$ and $y$  are the only vertices at distance $i$ from $a$ in $G$ to be used in any $a,b$-geodesic.
Let $E_i$ be the set of all edges in $S(G,a,b)$ having difference index $i$. If we note  that
$S(G,a,x)$,  $S(G,a,y)$, $S(G,x,b)$ and $S(G,y,b)$ are, respectively,  isomorphic to
$S(G_1,a,x)$,  $S(G_1,a,y)$, $S(G_2,x,b)$ and $S(G_2,y,b)$,  then it follows immediately from Theorem \ref{decomp}, part (a),  that
\[ S(G,a,b)/E_i \cong  S(G_1,a,x)\,\square\, S(G_2,x,b)\, \bigcup \, S(G_1,a,y)\,\square\, S(G_2,y,b).\]
From part (b) of that same theorem,  it follows directly that the edges connecting  the vertex disjoint components $S(G_1,a,x)\,\square\, S(G_2,x,b)$ and $S(G,a,y)\,\square\, S(G,y,b)$  form a matching. This completes the proof for Case 1.

\medskip

\noindent \textbf{Case (ii)} Either $d(a,x)\le d(a,y)$ and $d(x,b) < d(y,b)$,    or $d(a,x)<d(a,y)$ and $d(x,b) \le d(y,b)$.
Every $a,b$-geodesic in $G$ contains the vertex $x$, and the result follows directly from Theorem \ref{onesum}.

\medskip

\noindent \textbf{Case (iii)} Either $d(a,y)\le d(a,x)$ and $d(y,b) < d(x,b)$,   or $d(a,y)<d(a,x)$ and $d(y,b) \le d(x,b)$.
Every $a,b$-geodesic in $G$ contains the vertex $y$, and the result follows directly from Theorem \ref{onesum}.
\medskip

\noindent \textbf{Case (iv)} Consider first  when  $d(a,x) > d(a,y)$  and $d(x,b) < d(y,b)$.  Since $d(x,y)=1$,  we have that $d(a,x)=d(a,y)+1$ and $d(x,b)+1=d(y,b)$.
By Proposition \ref{concat},  the vertices of $S(G,a,b)$ which correspond to paths containing $x$ induce a subgraph isomorphic to $S(G_1,a,x)\square S(G_2, x,b)$,   and those which correspond to paths containing $y$ induce a subgraph isomorphic to $S(G_1,a,y)\square S(G_2, y,b)$.
Note that some $a,b$-geodesics contain the edge $e=xy$,  and hence the two induced subgraphs described above  have non-empty intersection.  Now let  $U$ be a vertex in $V(S(G,a,b))$  which corresponds to an $a,b$-geodesic containing $x$ and not $y$,  and let
$W$ be a vertex in $V(S(G,a,b))$  which corresponds to an $a,b$-geodesic containing $y$ but not $x$.   Then $U$ and $W$ differ in both index  $d(a,y)$ and index  $d(a,y)+1$ and hence are non-adjacent.
The case $d(a,x)< d(a,y)$ and $d(x,b) > d(y,b)$ is handled analogously. 
This completes the proof.
\end{proof}

  Note that in the proof of Theorem \ref{twosum},  the edge $e$ is used only in Case (iv).  Hence, if Case~(i),  (ii), or (iii) holds in the statement of that theorem,  then $S(G\setminus e,a,b) \cong S(G,a,b)$.
Also note that results similar to Theorem \ref{twosum} can be obtained by considering joining two graphs at two vertices  with no edges between the pair; or indeed joining graphs on more than two vertices.

\section{Shortest path graphs of girth at least $5$}\label{girth5}

In this section, we completely classify all shortest path graphs with girth $5$ or greater.  In the process,  we characterize precisely which cycles are shortest path graphs and we show that the claw is not a shortest path graph.
The following simple observation will be crucial.

\begin{proposition}\label{p3toc4}
Let $H$ be a shortest path graph.
Let $U_1,U_2,U_3$ be distinct vertices in $H$ such that $U_1U_2U_3$ is an induced path.
If the difference indices of $U_1U_2$ and $U_2U_3$ are $i$ and $j$, respectively, where $j\not\in\{i-1,i,i+1\}$,
then $H$ has an induced $C_4$ containing $U_1U_2U_3$.
\end{proposition}

\begin{proof}
Let $U_1 = av_1 \ldots v_pb$, $U_2 = av_1\ldots v_i'\ldots v_pb$, and $U_3 = av_1 \ldots v_{i-1}v_i'v_{i+1}\ldots v_j'\ldots v_p b$. 
Then there is a shortest path $U_4=av_1 \ldots v_j'\ldots v_p b$ in $G$, creating the $4$-cycle $(U_1,U_2,U_3,U_4)$.
\end{proof}

The next result says that any shortest path graph containing an induced odd cycle larger than a $3$-cycle must necessarily contain an induced $C_4$. Theorem~\ref{noClaw}  establishes the same result for induced claws.

\begin{lemma}\label{c4induced}
Let $H$ be a shortest path graph that contains an induced $C_k$ for odd $k > 3$. Then $H$ contains an induced $C_4$.
\end{lemma}

\begin{proof}
Let $(U_1, \ldots, U_k)$ be an induced $C_k$ with odd $k>3$ in $H$, and suppose that $H$ does not contain an induced $C_4$.
Let $i$ be the difference index of $U_1U_2$.
By Proposition~\ref{p3toc4}, the difference index of $U_2U_3$ is either $i-1$ or $i+1$.
In particular, if $i$ is odd then $U_2$ and $U_3$ differ at an even index, and if $i$ is even then $U_2$ and $U_3$ differ at an odd index. The same is true at every step, that is,  the parity  of the difference index  alternates around the cycle. This is impossible if $k$ is odd.
\end{proof}

In contrast, $C_3$ and every even cycle are shortest path graphs.

\begin{theorem}\label{ck}
$C_k$ is a shortest path graph, if and only if $k$ is even or $k=3$.
\end{theorem}

\begin{proof}
We have already seen that $C_3$ and $C_4$ are shortest path graphs, see Figure~\ref{moreExamples} and Corollary \ref{Cube}. From Lemma~\ref{c4induced}, it follows that $C_k$ is not a shortest path graph for odd $k > 3$.

We now construct a graph whose shortest path graph is $C_{2n}$. Define $G$  with  $2n+2$ vertices namely $ V(G)  = \{a,b,v_0,v_{1},\ldots,v_{n-1},v_{0}',v_{1}',\ldots,v_{n-1}'\}$
and edge set
\[
E(G)=\{av_{i}\,:\,i\in \mathbb{Z}_{n}\}\cup \{bv_{i}'\,:\,i\in\mathbb{Z}_{n}\} \cup \{v_{i}v_{i}': i \in \mathbb{Z}_{n}\} \cup \{v_{i}v_{i+1}'\,:\, i\in \mathbb{Z}_{n}\},
\]
where indices are calculated modulo $n$.
There are exactly $2n$ $a,b$-geodesics of $G$, namely the set $\{av_iv_i'b,av_iv'_{i+1}b\, :\, i=0,\ldots,n-1\}$. It is easy to check that the shortest path graph $S(G,a,b)=C_{2n}$. \\
\end{proof}

\begin{theorem}\label{noClaw}
If a shortest path graph $H$ has an induced claw, $K_{1,3}$, then $H$ must have a $4$-cycle containing two edges of the induced claw. In particular,
$K_{1,3}$ is not a shortest path graph.
\end{theorem}

\begin{proof}
Let $H$ be a shortest path graph that contains an induced claw with vertices $U_0,U_1,U_2,U_3$,
such that $U_0$ is adjacent to $U_1, U_2$, and $U_3$.  Let  $i_j$ be the difference index of
$U_0U_j$ for $j\in\{1, 2, 3\}$.  Since the claw is induced, these difference indices must be distinct.
Suppose that no three vertices of the claw are part of an induced $4$-cycle.
By Proposition \ref{p3toc4}, since $U_0, U_1, U_2$ is not a part of an induced $4$-cycle, it follows that $i_2= i_1 \pm 1$. Without a loss of generality let $i_2=i_1+1$. Similarly, because $U_0, U_1, U_3$ is not a part of an induced $4$-cycle, we have  $i_3= i_1 \pm 1$. Since the indices $i_j$ are distinct, it must be that $i_3=i_1-1$. By Proposition  \ref{p3toc4} it follows that  $U_0, U_2, U_3$ is in an induced $4$-cycle in $S(G,a,b)$. 
\end{proof}

An immediate consequence of Theorem~\ref{noClaw} is the following observation.
\begin{obs}
If $H$ is a tree and a shortest path graph, then $H$ is a path.
\end{obs}

Next we establish a characterization of when $C_k$ can be an induced subgraph of some shortest path graph.

\begin{theorem}\label{c5}
$C_k$ is an induced subgraph of some shortest path graph if and only if $k \neq 5$.
\end{theorem}

\begin{proof}
Assume to the contrary, that $S(G,a,b)$ contains an induced $C_5$, say  $\widetilde{U}=(U_1,U_2,U_3,U_4,U_5)$.
Consider the difference indices along the edges of the cycle. Every difference index that occurs must occur at least twice in order to return to the original shortest path.  Thus a 5-cycle can use at most 2 distinct difference indices.  Furthermore, if a difference index occurs twice in a row, say for  $U_1U_2$ and for $U_2U_3$, then the edge $U_1U_3$ is also in $S(G, a,b)$.
Therefore, $C_5$ is not an induced subgraph of a shortest path graph.

To finish the proof, we show that for any $k \neq 5$, $C_k$ is an induced subgraph of some shortest path graph.
In Theorem~\ref{ck} we saw that $C_k$ is itself a shortest path graph when $k=3$ or when $k$ is even. Thus, we only need to consider odd $k>6$.
Suppose that $k=2p+1$ and let $G_{2p+1}$ be a graph with vertex set $V(G_{2p+1})=\{a,b,v_1,v_2,\ldots,v_p,v_1',v_2',\ldots,v_p',v_1''\}$ and edge set
\begin{align*}
E(G_{2p+1}) &= \{av_1, av_1', bv_p, bv_p',av_1'',v_1''v_2',v_1''v_2\}\cup \{v_iv_{i+1},v_i'v_{i+1}',v_iv_{i+1}',v_i'v_{i+1}\,:\,i\in \{1,2,\ldots,n-1\}\}
\end{align*}
Then the following paths  of $G_{2p+1}$ induce a $C_{2p+1}$ in  $S(G_{2p+1}, a, b)$:
\[
\begin{array}{rrrrr}
av_1v_2v_3\cdots v_pb, \\
av_1'v_2v_3\cdots v_pb,  \\
av_1'v_2'v_3\cdots v_pb,\\
\ldots,  \\
av_1'v_2'v_3'\cdots v_p'b,  \\
av_1''v_2'v_3'\cdots v_p'b,\\
av_1''v_2v_3'\cdots v_p'b, \\
av_1''v_2v_3v_4'\cdots v_p'b, \\
\ldots, \\
av_1''v_2v_3\cdots v_{p-1} v_p'b, \\
av_1v_2v_3\cdots v_{p-1}v_p'b  \\
\end{array}
\]
\end{proof}

\begin{theorem} \label{girth5class}
Let $H$ be a graph with $\girth(H) \geq 5$. Then $H$ is a shortest path graph,  if and only if each nontrivial component of $H$ is a path or a cycle of even length greater than $5$.
\end{theorem}

\begin{proof}
If $\girth(H) \geq 5$, by Theorem~\ref{noClaw} there is no vertex $U \in V(H)$ with degree $\deg_H(U) \geq 3$.  Thus each vertex in $H$ must have degree $0$, $1$, or $2$. From this, it follows that every nontrivial component of $H$ is a path or cycle. By Lemma~\ref{c4induced}, any induced odd cycle forces an induced $C_4$. Therefore all of these cycles must have even length.

By Lemma~\ref{firstpath}, a path of any length is attained. In Theorem~\ref{ck}, it was shown how to construct a shortest path graph that is a cycle of even length.
Finally, by Theorem~\ref{disconnect}, the disjoint union of any set of shortest path graphs is again a shortest path graph.
\end{proof}

Now that shortest path graphs of  girth $5$ or more  have been characterized, a natural next step would be to work towards a characterization of girth $4$ shortest path graphs.  The prevalence of Cartesian products in shortest path graphs tends to indicate that $4$-cycles will play a large and challenging role in the study of these graphs.
We leave this challenge to a future paper and instead characterize  the shortest path graphs  of grid graphs, which have  particularly nice structure. We study these in the next section.

\section{Shortest Paths in Grid Graphs}\label{girth4}

An $m$-dimensional grid graph is the Cartesian product of $m$ paths, $P_{n_1} \square \cdots \square P_{n_m}$. We denote the vertices of $P_{n_1} \square \cdots \square P_{n_m}$ with the usual Cartesian coordinates on the $m$-dimensional lattice, so the vertex set is given by $V(P_{n_1} \square \cdots \square P_{n_m}) = \{(x_1,x_2,\ldots, x_m) : x_i \in \mathbb{Z}, 0 \leq x_1 \leq n_1, \ldots, 0 \leq x_m \leq n_m\}$. 
In what follows, we consider the geodesics between two diametric vertices of a grid graph, i.e., the shortest paths between the origin and the vertex $(n_1,\ldots,n_m)$. Because these two vertices will always be the vertices of consideration for grid graphs, we will denote the shortest path graph of $P_{n_1} \square \cdots \square P_{n_m}$ with respect to them simply by $S(P_{n_1} \square \cdots \square P_{n_m})$.
A two-dimensional grid graph and the diametric vertices under consideration are shown in Figure~\ref{Fig:G(n_1,n_2)}.

\begin{figure}[htb]%
\begin{center}
\includegraphics[scale=1]{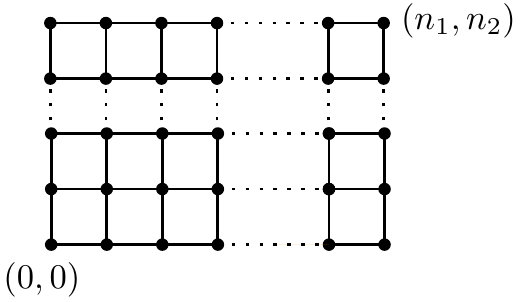}%
\end{center}
\caption{$P_{n_1}\square P_{n_2}$ }%
\label{Fig:G(n_1,n_2)}%
\end{figure}

For convenience of notation, we will consider the shortest paths in $P_{n_1} \square \cdots \square P_{n_m}$ as a sequence of moves through the grid in the following way. For $1 \leq i \leq m$, let $\mathbf{e_i}$ be the $i^{{\rm th}}$ standard basis vector in $\mathbb{R}^m$.
A move from a vertex $(x_1,\ldots,x_i,\ldots,x_m) \in V(P_{n_1} \square \cdots \square P_{n_m})$ in the $\mathbf{e_i}$ direction means that the next vertex along the path is $(x_1,\ldots,x_i+1,\ldots,x_m)$.

Note that a shortest path in $P_{n_1} \square \cdots \square P_{n_m}$ from $(0,\ldots,0)$ to $(n_1,\ldots,n_m)$ consists of exactly $N=\sum_{i=1}^{m} n_i$ moves, $n_i$ of which are in the $\mathbf{e_i}$ direction. Furthermore, observe that any $m$-ary sequence  of length $N$ in which the symbol $i$ occurs exactly  $n_i$ times corresponds to a geodesic in $P_{n_1} \square \cdots \square P_{n_m}$ and  that there are $\frac{\left(\sum_{i=1}^m n_i\right)!}{n_1!\cdots n_m!}$ such shortest paths.
Explicitly, a geodesic $U$  will be denoted by the  $m$-ary sequence $\widetilde{U}= s_1 \ldots  s_N\in \mathbb{Z}_m^N$, where $s_j = i \in \{1,\ldots,m\}$ if the $j^{{\rm th}}$ move in $U$ is in the $\mathbf{e_{i}}$ direction. In this way, the symbol $1$ corresponds to a move in the $\mathbf{e_1}$ direction, $2$ corresponds to a move in the $\mathbf{e_2}$ direction, etc. We will refer to $\widetilde{U}$ as the sequence representation of $U$ and use $B_{n_1,\ldots,n_m}\subset \mathbb{Z}_m^N$  
to denote the set of all  sequences with each $i$ occurring exactly $n_i$ times.

\begin{figure}[htb]%
\begin{center}
\includegraphics[scale=1]{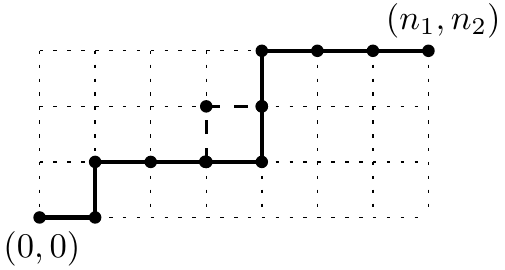}%
\end{center}
\caption{Adjacent paths in $P_{n_1} \square P_{n_2}$}%
\label{Fig:adjacent_paths}%
\end{figure}

Two shortest paths in $S(P_{n_1} \square \cdots \square P_{n_m})$ are adjacent if and only if they differ by a single vertex, i.e., if one path can be obtained from the other by swapping a single pair of two consecutive moves in different directions (see Figure~\ref{Fig:adjacent_paths}). Thus, if $U$ and $W$ are two shortest paths, then $U \sim W$ if and only if their sequence representations in  $B_{n_1,\ldots, n_m}$,
$\widetilde{U}$ and $\widetilde{W}$, respectively, have the forms
$\widetilde{U} = s_1 \ldots s_is_{i+1}\ldots s_N$ and $\widetilde{W} = s_1 \ldots s_{i+1}s_i\ldots s_N$ where $s_i \neq s_{i+1}$. 
It follows that
$UW\in E( S(P_{n_1} \square \cdots \square P_{n_m})) $
if and only if the  two sequences $\widetilde{U}, \widetilde{W} \in B_{n_1,\ldots,n_m}$ can be obtained from each  other by switching two different consecutive symbols.

The main result of this section is that $S(P_{n_1} \square \cdots \square P_{n_m})$ is  isomorphic to an induced subgraph of the integer lattice graph  $\mathbb{Z}^M$, where $M = \sum_{i=2}^{m}(i-1)n_i$.
To prove this result we define a mapping from $B_{n_1,\ldots,n_m}\subset \mathbb{Z}_m^N$ to coordinates in $\mathbb{Z}^M$.

 \begin{theorem}\label{Th:gridgraph}
The shortest path graph of $P_{n_1} \square \cdots \square P_{n_m}$ is isomorphic to an induced subgraph of the integer  lattice graph  $\mathbb{Z}^M$, where $M = \sum_{i=2}^{m}(i-1)n_i$.
\end{theorem}

\begin{proof}
Consider the $m$-dimensional grid graph  $P_{n_1} \square \cdots \square P_{n_m}$, with    $B_{n_1,\ldots,n_m}\subset \mathbb{Z}_m^N$
the set of all  $m$-ary sequences corresponding to its  geodesics.
Define a map $\phi: B_{n_1,\ldots,n_m}\rightarrow \mathbb{Z}^M.$
For a sequence $\widetilde{U} \in B_{n_1,\ldots,n_m}$, let

\[
\begin{array}{llllll}
\phi(\widetilde{U}):=    &(a_{121},\ldots, a_{12n_2}, & & &\\
	&\hspace{1ex} a_{131},\ldots, a_{13n_3},& a_{231},\ldots, a_{23n_3},&&\\
	&& \ldots &&\\
	&\hspace{1ex} a_{1m1},\ldots, a_{1mn_m},& a_{2m1},\ldots, a_{2mn_m},&\ldots&, a_{(m-1)m1},\ldots, a_{(m-1)mn_m}),
\end{array}
\]
where $a_{ijk}$ is the number of $i$'s following the $k^{{\rm th}}$ $j$ in $\widetilde{U}$. For example, if $\widetilde{U}=32121231 \in B_{3,3,2}$, then
$\phi(\widetilde{U}) =(3,2,1,3,1,3,0)$. Also, $\phi$ maps the sequence $1\cdots 1 2\cdots2 \cdots m\cdots m$ to the origin. Thus, $\phi$ maps $B_{n_1,\ldots,n_m}$
into a set of vectors $(a_{ijk}) \in \mathbb{Z}^M$ such that $ 2\leq j \leq m,  1 \leq i < j$, and $ 1 \leq k \leq n_j$, in the order indicated.
Note that for all $i,j,k$, $a_{ijk} \leq n_i$ since there are at most $n_i$ $i$'s following a $j$, and $a_{ijk} \geq a_{ij(k+1)}$ since at least as many $i$'s appear after the
 $k^{{\rm th}}$ $j$ than after the $(k+1)^{{\rm st}}$ $j$.

To see that $\phi$ is injective, consider two distinct sequences $\widetilde{U} = s_1\ldots s_N$ and $\widetilde{W} = s_1'\ldots s_N'$ in $B_{n_1,\ldots,n_m}$, and denote their images under $\phi$ by $A$ and $A'$, respectively. Let $r$ be the first index where the entries of $\widetilde{U}$ and $\widetilde{W}$ differ.
Without loss of generality, assume that $s_r = j > s_r' = i$ and that $s_r$ is the $k^{{\rm th}}$ $j$ occurring in $\widetilde{U}$. Then, the $k^{{\rm th}}$ $j$ in
 $\widetilde{W}$ will appear after $s_r' = i$, so the number of $i$'s in $\widetilde{W}$ following the $k^{{\rm th}}$ $j$ will be at least one less as compared to the sequence $\widetilde{U}$. Therefore, if $a_{ijk}$ and $a_{ijk}'$ are the $ijk$-components of $A$ and $A'$, respectively, then $a_{ijk} > {a}_{ijk}'$, showing
  $\phi(\widetilde{U}) \neq \phi(\widetilde{W})$.

To finish the proof, we need to show that $\phi$ preserves adjacency. Suppose $\widetilde{U}$ and $\widetilde{W}$ are adjacent sequences in $B_{n_1,\ldots,n_m}$. Then $\widetilde{U}$ and $\widetilde{W}$ have the forms $\widetilde{U}=s_1\ldots s_rs_{r+1} \ldots s_N$ and $\widetilde{W}=s_1\ldots s_{r+1}s_r \ldots s_N$, where $s_r \neq s_{r+1}$. Let $ s_{r+1}=j> s_r = i$.
Now, suppose that $s_{r+1}$ is the $k^{{\rm th}}$ $j$ appearing in $\widetilde{U}$. Then the only difference in the vectors $\phi(\widetilde{U})$ and $\phi(\widetilde{W})$ is that the $ijk$-component of $\phi({\widetilde{W}})$ is increased by one unit, so $\phi(\widetilde{U})$ and $\phi({\widetilde{W}})$ are adjacent vertices in $\mathbb{Z}^M$.
To see that $\phi^{-1}$ also preserves adjacency, consider two adjacent vertices, $A$ and $A'$, in the image of $\phi$. Let $a_{ijk}$ be the $ijk$-component of $A$, and without loss of generality, assume that ${A'}$ is obtained from $A$ by increasing $a_{ijk}$ to $a_{ijk}+1$. Because ${A'}$ is in the image of $\phi$, it follows that the symbol directly preceding the $k^{{\rm th}}$ $j$ in $\phi^{-1}(A)$ is $i$. To see this, first note that there must be at least one $i$ preceding the $k^{{\rm th}}$ $j$. Otherwise, $a_{ijk} = n_i$ and cannot be increased. If there is any other symbol between the $k^{{\rm th}}$ $j$ and the $i$ preceding it, other components of $A$ must be changed in order to increase $a_{ijk}$.
However, $A$ and $A'$ have only one different component. Thus, $\phi^{-1}(A')$ can be obtained from $\phi^{-1}(A)$ by switching the $k^\text{th}$ $j$ and the $i$ directly preceding it. Therefore, $\phi^{-1}(A)$ and $\phi^{-1}(A')$ are adjacent vertices in $B_{n_1,\ldots,n_m}$.
\end{proof}

The shortest path graph of a two-dimensional grid graph $P_{n_1} \square P_{n_2}$ is particularly easy to characterize, as demonstrated in the following corollary. First, we need an additional definition.
The staircase graph $S_{n_1,n_2}$ is an induced subgraph of the grid graph on the integer lattice $\mathbb{Z}^{n_2}$. 
$S_{n_1,n_2}$ has vertex set
$V(S_{n_1,n_2}) = \{(a_1,\ldots,a_{n_2}) : a_k \in \mathbb{Z}, n_1\geq a_1 \geq a_2 \geq \cdots \geq a_{n_2} \geq 0\}$ (see Figure \ref{Fig:Stairs}).

\begin{figure}[htb]
\begin{center}
\includegraphics[scale=1]{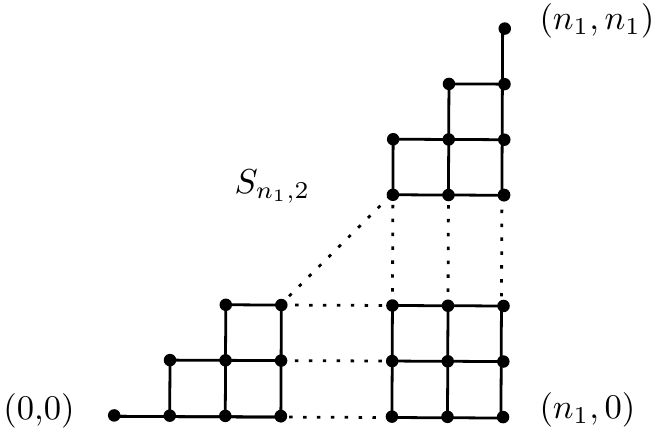}%
\includegraphics[scale=1]{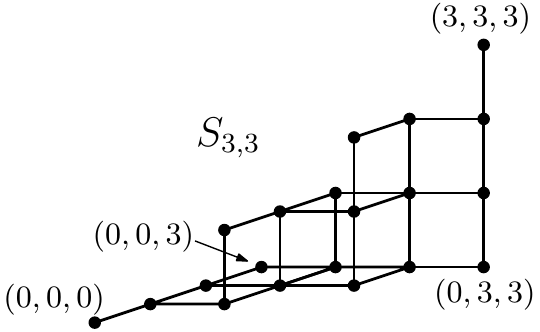}
\end{center}
\caption{The staircase graphs $S_{n_1,2}$ (left) and $S_{3,3}$ (right)}%
\label{Fig:Stairs}%
\end{figure}

\begin{corollary}
The shortest path graph of $P_{n_1} \square P_{n_2}$ is isomorphic to the staircase graph $S_{n_1,n_2}$.
\end{corollary}
\begin{proof}
We have seen that the shortest path graph $S(P_{n_1}\square P_{n_2})$ can be described as a graph on the  set of binary strings $B_{n_1,n_2}$. Furthermore, from the proof of Theorem \ref{Th:gridgraph}, the mapping $\phi:B_{n_1,n_2} \rightarrow \mathbb{Z}^{n_2}$ defined by $\phi(\widetilde{U}) := (a_1, a_2\ldots,a_{n_2})$, where $a_k$ is the number of $1$'s following the $k^{{\rm th}}$ occurrence of $2$ in $\widetilde{U}$, is injective and adjacency preserving. Therefore, this corollary follows if we can show $\phi(B_{n_1,n_2}) = V(S_{n_1,n_2})$.

Let $\widetilde{U} \in B_{n_1,n_2}$ and let $\phi(\widetilde{U}) = (a_1, a_2\ldots,a_{n_2})$. From the definition of $\phi$, it follows that $n_1\geq a_k \geq a_{k+1} \geq 0$ for all $k \in\{ 1,\ldots,n_2-1\}$, since the number of $1$'s following the $k^{{\rm th}}$ $2$ is greater than or equal to the number of $1$'s following the $(k+1)^{{\rm st}}$ $2$. Thus, $\phi(\widetilde{U}) \in V(S_{n_1,n_2})$.
Conversely, for any  $A=(a_1,\ldots,a_{n_2}) \in V(S_{n_1,n_2})$, let $\widetilde{U}$ be the sequence in $B_{n_1,n_2}$ that has exactly $a_k$ $1$'s following the $k^{{\rm th}}$ $2$ for $k\in\{1,\ldots,n_2\}$. Then, $\phi(\widetilde{U}) = A$ showing $V(S_{n_1,n_2}) \subseteq \phi(B_{n_1,n_2})$.
\end{proof}

\begin{remark} Theorem \ref{Th:gridgraph} implies that the dimension of the lattice graph $\mathbb{Z}^M$ of which $S(P_{n_1} \square \cdots \square P_{n_m})$ is an induced subgraph depends on the ordering of $n_1,\ldots, n_m$. Since  $M=\sum_{i=2}^{m}(i-1)n_i$, the least value for  $M$ will occur when $n_1,\cdots, n_m$ are  listed in decreasing order.
\end{remark}

It is a direct consequence of our discussion on grid graphs that the path of length $k$ is the shortest path graph of $P_k\square P_1$.
\begin{corollary}\label{pathThm}
For the grid $G=P_k\square P_{1}$, $S(G)\cong P_k$.
\end{corollary}

Earlier we made the comment that two base graphs may produce the same shortest path graph. In fact, even two reduced graphs can have the same shortest path graph, e.g., the  graphs $P_k\square P_1$ and $G_k$ given in  Lemma \ref{firstpath} have the same shortest path graph 
 yet they are reduced and non-isomorphic.

Another special grid graph is the $m$-dimensional  hypercube, $Q_m=P_1 \square \cdots \square P_1$.
We shall observe  in Proposition \ref{Th:Cayley}, that $S(Q_m)$ is isomorphic to a Cayley graph of the symmetric group $S_m$.

We first  recall some material from elementary group theory and algebraic graph theory. See  \cite{Bacher, Fraleigh, Godsil} for more detail.
Let $(\Gamma, \cdot)$ be a group. 
 Let $S$ be a generating set of $\Gamma$ that does not contain the identity element and such that  for each $g\in S$, $g^{-1}$ is also in $S$. The {\em Cayley graph} of $\Gamma$ with generating set $S$, denoted by  $\mbox{Cay}(\Gamma; S)$,  is the graph whose vertices are the elements of $\Gamma$, and which has an edge between two vertices $x$ and $y$ if and only if $x\cdot s = y$ for some $s\in S$.

The symmetric group $S_m$ is the group whose elements are the permutations on the set $\{1, 2, \ldots, m\}$. An element of  $S_m$ is a bijection from  the set $\{1, 2, \ldots, m\}$ to itself.  Denote by $s_1s_2 \ldots s_m$ the permutation
$\sigma$ given by $\sigma(i)=s_i$, $1\le i \le m$.
 The group operation in $S_m$ is the composition of permutations defined by $(\sigma\tau )(j)=\sigma(\tau(j))$, for $\sigma, \tau \in S_m$,  $1\le j\le m$.
An \emph{adjacent transposition} is a permutation $\tau_i$ such that
\[\tau_i(i) = i+1, \tau_i(i+1) = i \text{, and } \tau_i(k) = k  \text{ for } k\not\in\{i,i+1\}.\]
 It is well known that every permutation can be represented as the composition of finitely many adjacent transpositions. Therefore, the set of adjacent transpositions, $T$, generates $S_m$. We also note that each adjacent transposition is its own inverse. Hence, we can define the Cayley graph of $S_m$ with generating set $T$, $\mbox{Cay}(S_m; T )$.\footnote{This Cayley graph was studied by Bacher \cite{Bacher} and is also known, in other contexts, as a Bubble Sort Graph.}

To gain some insight into the structure of  $\mbox{Cay}(S_m; T)$,  consider the effect of the composition of an element $\sigma \in S_m$ with an adjacent transposition.
Let $\sigma = s_1s_2 \ldots s_m$ and $1\le i\le m-1$. Then
\[\sigma\tau_i(i) = \sigma(i+1),\, \sigma\tau_i(i+1) = \sigma(i), \text{ and } \sigma\tau_i(k) =\sigma( k), \text{ if } k \not\in\{i,i+1\},\]
or simply
\[\sigma\tau_i = s_1s_2 \ldots s_{i+1}j_i \ldots s_m.\]
Thus, the effect of the composition $\sigma\tau_i$ is the switching of the two consecutive elements $s_i$ and $s_{i+1}$  in $\sigma$.
We can conclude that the neighborhood of a vertex $\sigma$ in $\mbox{Cay}(S_m; T )$ is the collection of all $(m-1)$ permutations on $\{1,\ldots,m\}$ obtained from $\sigma$ by interchanging two consecutive elements.

Now  consider the $m$-dimensional hypercube, $Q_m=P_1 \square \cdots \square P_1$.
The vertices of  $P_1 \square \cdots \square P_1$ correspond to all binary strings of length $m$.  A shortest path in $Q_m$ is a sequence of exactly one move in each direction. 
In the  discussion preceding Theorem \ref{Th:gridgraph} we introduced a  correspondence between the geodesics of  
$P_{n_1} \square \cdots \square P_{n_m}$ and the set of sequences of moves $B_{n_1,\ldots,n_m}$.
For $Q_m$, this is a bijection between the geodesics in $P_1 \square \cdots \square P_1$ and the sequences in $B_{1,\ldots,1}$. Since  in each geodesic a vertex coordinate changes exactly once, $B_{1,\ldots,1}$ coincides with the set of permutations on $\{1, 2, \ldots, m\}$.  The sequence representation
of $U$ is denoted by $\widetilde{U} = s_1s_2 \ldots s_m$,  where each element of $\{1, \ldots, m\}$ appears precisely once.
 Recall that we defined two sequences in $B_{1,\ldots,1}$ as adjacent, if and only if one can be obtained from the other by switching two (different) consecutive symbols. This is equivalent to two permutations being adjacent if and only if one can be obtained from the other using an adjacent transposition. Thus, the set $B_{1,\ldots,1}$ together with the adjacency relation is isomorphic to $\mbox{Cay}(S_m; T )$.  Hence we observe:

\begin{proposition}\label{Th:Cayley} Let $S_m$ be the symmetric group,  let $T$ be the set of adjacent transpositions, and
let $a$ and $b$ be diametric vertices on $Q_m$.
 Then $S(Q_m,a,b) \cong \mbox{Cay}(S_m; T )$.
\end{proposition}

Although the function $\phi$ introduced in the proof of Theorem \ref{Th:gridgraph} is not needed in Proposition \ref{Th:Cayley},  that function has an interesting interpretation in the case of $Q_m$.  Here the domain
of $\phi$ is $B_{1,\ldots,1}$,  which as we have discussed,  is isomorphic to $S_m$,  the set of permutations of $\{1,2, \ldots, m\}$.  Referring to the definition of $\phi$ in the proof of Theorem \ref{Th:gridgraph},  and using the notation introduced there, the image of  $\phi$ is a sequence whose elements are $a_{ijk}$,  where $1\le i <j \le m$ and $1\le k \le n_j$.   In the case of $Q_m$, $n_j =1$ for each $j$.
Hence, for a sequence $\widetilde{U}\in B_{1, \ldots, 1}$,
corresponding to a permutation $s_1s_2 \cdots s_m$,  it makes sense to simplify our notation to 
\[
\phi(\widetilde{U}):=    (a_{12}, a_{13}, a_{23},  \ldots, a_{1m},  a_{2m}, \ldots  , a_{(m-1)m}),
\]
where $a_{ij}$ is equal to the number of $i$'s following the $1^{\rm st}$ (and only) $j$ in $\widetilde{U}$.
From  Theorem \ref{Th:gridgraph}, the length of the sequence
$\phi(\widetilde{U})$ is  $M = \sum_{i=2}^{m}(i-1)n_i= \sum_{i=1}^m(i-1)= {m \choose 2}$.   Since every element in $\{1,2, \ldots ,m\}$ occurs precisely once in the permutation $s_1s_2 \ldots s_m$,  we have that for every pair $i,j$ with
$1\le i<j\le m$,
 \[
    a_{ij} = \begin{cases}
        0  & \text{if  $i$ occurs before $j$ and}\\
        1  & \text{if $i$ occurs after $j$.}
\end{cases}
  \]
In the case of $Q_m$, one may interpret $\phi(\widetilde{U})$  as the edge set of a complete directed graph on $m$ vertices as follows.  For each pair of vertices $i,j$ with  $1\le i<j \le m$,  the  edge  $ij$ is oriented from $i$ to $j$ 
if $a_{ij}=0$,   and  from $j$ to $i$ if $a_{ij}=1$. 
A complete directed graph having this transitive property is called a {\em transitive tournament}.    We conclude that for the hypercube $Q_m$,  the image of 
$\phi$ corresponds precisely to the set of $m!$ transitive tournaments.

\bibliographystyle{amsplain}
\bibliography{vdec}

\end{document}